\documentclass[12pt]{article}
\usepackage{a4,enumerate,latexsym,amssymb,amsthm,amsmath,amssymb,eucal}

\newtheorem{theorem}{Theorem}
\newtheorem{lemma}[theorem]{Lemma}

\newtheorem{corollary}[theorem]{Corollary}
\newtheorem{conjecture}[theorem]{Conjecture}

\newtheorem{observation}[theorem]{Observation}

\newcommand\size[1] {\left|{#1}\right|}
\newcommand\eps \varepsilon
\newcommand\sm \setminus
\newcommand\Set[2] {\left\{{#1}:\,{#2}\right\}}
\newcommand\Setx[1] {\left\{{#1}\right\}}
\newcommand\sub {\subseteq}
\newcommand{\eind}[2]{#1[{#2}]}
\newcommand{\union}[2]{{#2}^{(#1)}}
\newcommand\conn[1] {\eta({#1})}
\newcommand{\subc}[2]{#1\left[#2\right]}

\newcommand\rankx[2] {\mathrm{rank}_{#1}({#2})}
\newcommand\edgedom[1] {\gamma^E({#1})}
\newcommand\pathdom[1] {\gamma^\mathrm{v}({#1})}
\newcommand\poly[1] {\|#1\|}
\newcommand\CC {{\mathcal C}}
\newcommand\MM {{\mathcal M}}
\newcommand\NN {{\mathcal N}}
\newcommand\II {{\mathcal I}}
\newcommand\KK {{\mathcal K}}
\newcommand\RR {{\mathbb R}}

\newcommand{\arbf}[1]{\Upsilon_f(#1)}
\newcommand{\arb}[1]{\Upsilon(#1)}

\title{\textbf{Covering a graph by forests\\and a matching}}
\author{Tom\'{a}\v{s} Kaiser$^1$ \and 
  Micka\"{e}l Montassier$^2$ \and
  Andr\'{e} Raspaud$^2$}

\date{}

\begin{document}
\maketitle
\footnotetext[1]{Department of Mathematics and Institute for
  Theoretical Computer Science, University of West Bohemia,
  Univerzitn\'{\i}~8, 306~14~Plze\v{n}, Czech Republic. E-mail:
  \texttt{kaisert@kma.zcu.cz}. Supported by grant GA\v{C}R 201/09/0197
  of the Czech Science Foundation and by project 1M0545 and Research
  Plan MSM 4977751301 of the Czech Ministry of Education. This
  research was done during the first author's visit to LaBRI.}
\footnotetext[2]{LaBRI, Universit\'{e} Bordeaux I, 351, cours de la
  Lib\'{e}ration, 33405~Talence cedex, France. E-mail:
  \{\texttt{raspaud,montassier}\}\texttt{@labri.fr}.}

\begin{abstract}
  We prove that for any positive integer $k$, the edges of any graph
  whose fractional arboricity is at most $k + 1/(3k+2)$ can be
  decomposed into $k$ forests and a matching.
\end{abstract}

%%%%%%%%%%%%%%%%%%%%%%%%%%%%%%%%%%%%%%%%%%%%%%%%%%%%%%%%%%%%%%%%%%%%%%

\section{Introduction}\label{sec:intro}

The \emph{arboricity} $\arb G$ of a graph $G$ is the least number $k$
such that the edge set of $G$ can be covered by $k$ forests. A
classical result of Nash-Williams~\cite{NW:decomposition} states that
a trivial lower bound to arboricity actually gives the right value:
\begin{theorem}\label{t:arb}
  For any graph $G$,
  \begin{equation*}
    \arb G = \max_H \Bigl\lceil
    \frac{\size{E(H)}}{\size{V(H)}-1} 
    \Bigr\rceil,
  \end{equation*}
  where the maximum is taken over all subgraphs $H$ of $G$.
\end{theorem}
(Here and in the rest of the paper, we write $V(H)$ and $E(H)$ for the
vertex set and the edge set of a graph $H$, respectively, and the
graphs may contain loops and parallel edges. For any graph-theoretical
notions not defined here, we refer the reader to
Diestel~\cite{Die:graph}.)

Payan~\cite{P:graphes} defined the \emph{fractional arboricity} $\arbf G$ of
$G$ by
\begin{equation*}
  \arbf G = \max_{H\sub G} \frac{\size{E(H)}}{\size{V(H)}-1}.
\end{equation*}
Thus, $\arb G = \lceil\arbf G\rceil$, and one may ask whether $\arbf
G$ is a finer measure of the properties of $G$ than $\arb G$. In
particular, suppose that $\arbf G = k + \eps$ for some integer $k$ and
small $\eps > 0$. By Theorem~\ref{t:arb}, $E(G)$ can be covered by
$k+1$ forests, but it is natural to ask whether one of these forests
can be restricted to have, for instance, bounded maximum degree or
bounded maximum component size. The problem was studied by Montassier
et al.~\cite{MORZ:covering} and in two cases, an affirmative answer
was obtained. These are summarized in the following theorem which
improves earlier results on decompositions of planar
graphs~\cite{BKSY:decomposing,G:covering}:

\begin{theorem}\label{t:43}
  Let $G$ be a graph.
  \begin{enumerate}[\quad(i)]
  \item If $\arbf G \leq \tfrac43$, then $E(G)$ can be covered by a
    forest and a matching.
  \item If $\arbf G \leq \tfrac32$, then $E(G)$ can be covered by two
    forests, one of which has maximum degree at most 2.
  \end{enumerate}
\end{theorem}

In~\cite{MORZ:covering}, a general conjecture was proposed which would
include Theorem~\ref{t:43} as a special case:

\begin{conjecture}\label{conj:dragon}
  Let $k$ and $d$ be positive integers. If $G$ is a graph with $\arbf
  G \leq k + \frac{d}{k+d+1}$, then $E(G)$ can be decomposed into
  $k+1$ forests, one of which has maximum degree at most $d$.
\end{conjecture}

A. V. Kostochka and X. Zhu (personal communication) proved
Conjecture~\ref{conj:dragon} for $k = 1$ and $3\leq d \leq 6$. For $k
\geq 2$ or $d\geq 7$, the conjecture is open.

Another partial result of~\cite{MORZ:covering} toward
Conjecture~\ref{conj:dragon} is the following:
\begin{theorem}\label{t:weaker}
  If $G$ has fractional arboricity $\arbf G \leq k + \eps$, where
  $k$ is an integer and $0\leq \eps < 1$, then $E(G)$ can be
  covered by $k$ forests and a graph of maximum degree at most $d$,
  where
  \begin{equation*}
    d = \Bigl\lceil \frac{(k+1)(k-1+2\eps)}{1-\eps} \Bigr\rceil.
  \end{equation*}
\end{theorem}
Theorem~\ref{t:weaker} provides a value of $d$ for any choice of
$\eps$, but it does not ensure a suitable value of $\eps$ for an
arbitrary $d$. In particular, for $k \geq 2$ it leaves open the
question whether there is $\eps = \eps(k)$ such that the edges of any
graph with fractional arboricity at most $k + \eps$ can be covered by
$k$ forests and a matching. In the present paper, we answer this
question in the affirmative:

\begin{theorem}\label{t:main}
  Let $k \geq 1$ be an integer and $G$ a graph with $\arbf G \leq k +
  \tfrac1{3k+2}$. Then $E(G)$ can be decomposed into $k$ forests and a
  matching.
\end{theorem}

Our proof is based on an extension of the matroid intersection
theorem~\cite{Edm:minimum} due to Aharoni and
Berger~\cite{AB:intersection}. The structure of the paper is as
follows. In Section~\ref{sec:matroid}, we recall the necessary notions
of matroid theory. Section~\ref{sec:complex} gives an overview of the
topological preliminaries. The pieces are assembled in
Section~\ref{sec:proof}, where we prove Theorem~\ref{t:main}.

%%%%%%%%%%%%%%%%%%%%%%%%%%%%%%%%%%%%%%%%%%%%%%%%%%%%%%%%%%%%%%%%%%%%%%

\section{Matroids}
\label{sec:matroid}

The purpose of this section to introduce the relevant terminology and
facts from matroid theory. For more details, the reader may consult
the book of Oxley~\cite{Oxl:matroid} or Part IV of
Schrijver~\cite{Sch:combinatorial}.

A \emph{matroid} is a pair $(E,\II)$, where $E$ is a finite set and
$\II$ is a nonempty collection of subsets of $E$ satisfying the
following axioms:
\begin{enumerate}[\quad(M1)]
\item if $A\sub B \sub E$ and $B\in \II$, then $A\in\II$, and
\item if $A,B \in \II$ and $\size A < \size B$, then for some $x\in
  B\sm A$, $A\cup x \in \II$. 
\end{enumerate}
(For brevity, we write $A\cup x$ in place of $A\cup\Setx x$, and $A\sm
x$ instead of $A \sm\Setx{x}$.)

Let $\MM = (E,\II)$ be a matroid. The sets in $\II$ are called
\emph{independent sets} of $\MM$ (the other subsets of $E$ being
\emph{dependent}), and $\MM$ is said to be a matroid \emph{on} $E$.

It is easy to prove from (M2) that all inclusionwise maximal subsets
of a set $X\sub E$ that are independent in $\MM$ have the same
cardinality. This cardinality is called the \emph{rank} of $X$ in
$\MM$ and denoted by $\rankx \MM X$. By definition, the rank of $\MM$
is $\rankx \MM E$. Any independent set of size $\rankx \MM E$ is a
\emph{base} of $\MM$.

The \emph{dual matroid} $\MM^*$ of $\MM$ is a matroid on $E$ whose
independent sets are all the subsets of $E$ that are disjoint from
some base of $\MM$. Thus, the bases of $\MM^*$ are precisely the
conplements of the bases of $\MM$. The rank function of the dual
matroid is given by the following lemma (see~\cite[Proposition
2.1.9]{Oxl:matroid}):

\begin{lemma}\label{l:dual}
  If $\MM$ is a matroid on $E$ and $X$ is a subset of $E$, then the
  rank of $X$ in the dual matroid $\MM^*$ is
  \begin{equation*}
    \rankx{\MM^*} X = \size X + \rankx\MM{E \sm X} - \rankx\MM E.
  \end{equation*}
\end{lemma}

Each graph $G$ has an associated matroid, the \emph{cycle matroid}
of $G$. This is a matroid on $E(G)$ and its independent sets are the
edge sets of forests in $G$. For $X\sub E(G)$, let $\eind G X$ be the
subgraph of $G$ induced by the edge set $X$ (that is, its vertices are
all the vertices incident with an edge of $X$, and its edge set is
$X$). The rank function of $X$ is easily interpreted in terms of
$\eind G X$:

\begin{lemma}\label{l:rank}
  Let $\MM$ be the cycle matroid of a graph $G$ and $X\sub
  E(G)$. If the subgraph $\eind G X$ has $n(X)$ vertices and $c(X)$
  components, then
  \begin{equation*}
    \rankx \MM X = n(X) - c(X).
  \end{equation*}
\end{lemma}

Theorem~\ref{t:arb} has a natural proof using matroid theory, based on
the following important result of Nash-Williams~\cite{NW:application}
(see also~\cite[Proposition~12.3.1]{Oxl:matroid}):

\begin{theorem}[Matroid union theorem]\label{t:union}
  Let $\MM$ and $\NN$ be matroids on $E$. Let $\II$ be the collection
  of all sets $I\cup J$, where $I$ is an independent set of $\MM$ and
  $J$ is an independent set of $\NN$. Then $(E,\II)$ is a matroid
  and the rank of a set $X\sub E$ in this matroid equals
  \begin{equation*}
    \min_{T\sub X} \Bigl(\size{X\sm T} + \rankx \MM T + \rankx \NN T\Bigr).
  \end{equation*}
\end{theorem}
The matroid from Theorem~\ref{t:union} is called the \emph{union} of
$\MM$ and $\NN$ and is denoted by $\MM\vee\NN$. 

Let $\union k \MM$ be the union of $k$ copies of $\MM$. The Matroid
union theorem implies that the rank function of this matroid is as
follows:
\begin{corollary}\label{cor:mk}
  The rank of a set $X\sub E$ in $\union k \MM$ is
  \begin{equation*}
    \rankx{\union k \MM} X = \min_{T\sub X} \Bigl( \size{X \sm T} + 
    k \cdot \rankx\MM T \Bigr).
  \end{equation*}
\end{corollary}

Closely related to Theorem~\ref{t:union} is the following result of
Edmonds~\cite{Edm:minimum} (see also~\cite[Theorem 12.3.15]{Oxl:matroid}):

\begin{theorem}[Matroid intersection theorem]\label{t:m-inters}
  Let $\MM$ and $\NN$ be matroids on $E$. The maximum size of a subset
  of $E$ that is independent in both $\MM$ and $\NN$ equals
  \begin{equation*}
    \min_{X\sub E} \Bigl(\rankx\MM X + \rankx\NN {E\sm X}\Bigr).
  \end{equation*}
\end{theorem}

A \emph{circuit} of $\MM$ is any inclusionwise minimal dependent
subset of $E$. A set $X\sub E$ is a \emph{flat} of $\MM$ if for every
$x\in E\sm X$, $\rankx\MM{X\cup x} = \rankx\MM X + 1$. We will need the
following lemma which relates circuit and flats
(see~\cite[Proposition~1.4.10]{Oxl:matroid} for a proof):

\begin{lemma}\label{l:circuit}
  If $X$ is a flat of $\MM$ and $e\in E$ is contained in a circuit $C$
  such that $C \sub X\cup e$, then $e\in X$.
\end{lemma}

%%%%%%%%%%%%%%%%%%%%%%%%%%%%%%%%%%%%%%%%%%%%%%%%%%%%%%%%%%%%%%%%%%%%%%

\section{Complexes}
\label{sec:complex}

In this section, we review the topological machinery needed in our
proof. A more complete account can be found in Section 2
of~\cite{AB:intersection}. A standard reference on topological methods
in combinatorics is Bj\"{o}rner~\cite{Bjo:topological}.

A \emph{simplicial complex} (or just \emph{complex}) on a finite set
$E$ is any nonempty collection $\CC$ of subsets of $E$ such that if
$A\sub B \in \CC$, then $A \in \CC$. The subsets belonging to $\CC$
are called the \emph{faces} (or \emph{simplices}) of $\CC$.

Any complex $\CC$ has an associated geometric realization $\poly\CC$
called the \emph{polyhedron} of $\CC$. This is a topological space
obtained as follows. To each $e\in E$ contained in some face of $\CC$,
assign a vector $v_e$ in $\RR^{\size E}$ in such a way that all the
vectors $v_e$ are linearly independent. Every face $A$ of $\CC$ then
has an associated geometric simplex $\sigma_A$, namely the convex hull
of the set $\Set{v_e}{e\in A}$. The polyhedron of $\CC$ is obtained as
the union of all the simplices $\sigma_A$, where $A$ ranges over
$\CC$.

Next, we recall the notion of connectivity of complexes and
topological spaces in general. For $d\geq 0$, a topological space $X$
is \emph{$d$-connected} if every continuous mapping $f$ from the
$d$-dimensional sphere $S^d$ to $X$ can be extended to the
$(d+1)$-dimensional closed ball $B^{d+1}$ in a continuous way. Every
nonempty space is considered to be ($-1$)-connected. The
\emph{connectivity} of a nonempty space $X$ is the largest integer $k$
such that $X$ is $d$-connected for all $d$, where $-1 \leq d \leq
k$. (If $X$ is $d$-connected for all integers $d \geq -1$, then its
connectivity is infinite.) The property of being 0-connected is
equivalent to the usual arcwise connectedness of $X$. For higher $d$,
$d$-connected spaces can be intuitively thought of as those which have
no `holes' of dimension less than $d$.

The connectivity of a complex $\CC$ is defined as the connectivity of
its polyhedron $\poly\CC$. For technical reasons, it is useful to work
with a slight modification of this parameter, denoted by $\conn\CC$
and defined as the connectivity of $\CC$ plus 2.

Matroids can be viewed as complexes of a special type: if $\MM$ is a
matroid on a set $E$, then the independent sets of $\MM$ form a
complex on $E$. While most properties of matroids do not carry over to
the more general world of complexes, one important matroid-theoretical
result that has a partial extension to complexes is
Theorem~\ref{t:m-inters}. This extension is due to Aharoni and
Berger~\cite[Theorem 4.5]{AB:intersection} and we will use it in the
following formulation:

\begin{theorem}\label{t:inters}
  Let $\NN$ be a matroid on $E$ and let $\CC$ be a complex whose
  vertex set is also $E$. If 
  \begin{equation}\label{eq:inters}
    \conn{\subc\CC X} \geq \rankx\NN E - \rankx\NN{E\sm X}
  \end{equation}
  for every $X\sub E$ which is the complement of a flat of $\NN$, then
  $\NN$ has a base which is a face of $\CC$.
\end{theorem}

It can be shown~\cite{BKL:homotopy} that if $\CC$ is the complex of
independent sets of a matroid $\MM$, then $\conn\CC$ equals the rank
of $\MM$ (unless the dual of $\MM$ contains a loop, in which case
$\conn\CC$ is infinite). From this, one can easily derive that
Theorem~\ref{t:inters} implies the nontrivial direction of
Theorem~\ref{t:m-inters}.

Another class of complexes that is relevant in this paper is that of
independence complexes. The \emph{independence complex} $\II(H)$ of a
graph $H$ is a complex on $V(H)$ whose faces are the independent sets
of $H$ (that is, sets $I$ such that the induced subgraph of $H$ on $I$
has no edges). When $H$ is the line graph of a graph $G$, the
independent sets of $H$ are the matchings of $G$ and this construction
produces the \emph{matching complex} of $G$.

Since the connectivity of a complex is in general difficult to
establish, it is very useful that there are several results relating
the connectivity of an independence complex $\II(H)$ to the properties
of the graph $H$. Typically, these properties concern some variant of
the notion of domination in $H$ (a useful overview is given
in~\cite[Section 2]{ABZ:independent}). In our case, the bound involves
the edge-domination number which we recall next.

A set $D$ of edges of the graph $H$ is said to \emph{dominate} $H$ if
for every vertex $v$ of $H$, $v$ or at least one of its neighbors is
incident with an edge of $D$. The least cardinality of a set of edges
dominating $H$ is the \emph{edge-domination number} $\edgedom H$ of
$H$. (If $H$ contains an isolated vertex, then there is no dominating
set of edges and $\edgedom H$ is defined to be infinite.) The
following result is implicitly proved in several papers on
independence complexes (for references, see~\cite{ABZ:independent}):

\begin{theorem}\label{t:edgedom}
  If $H$ is a graph, then
  \begin{equation*}
    \conn{\II(H)} \geq \edgedom H.
  \end{equation*}
\end{theorem}

If we specialize this result to line graphs (and matching complexes),
we obtain a notion previously used in~\cite{Kai:note}. A \emph{2-path}
in a graph is a path of length 2. A set $P$ of 2-paths in $G$
\emph{dominates} a set $F$ of edges if every edge of $F$ is incident
with a 2-path in $P$. The \emph{2-path domination number} $\pathdom G$
is the minimum size of a set of 2-paths dominating $E(G)$ (or $\infty$
if $G$ contains a component with exactly one edge). Since $\pathdom G$
is equal to $\edgedom{L(G)}$, we have the following observation:

\begin{observation}\label{obs:pathdom}
  If $G$ is a graph and $\CC$ is its matching complex, then
  \begin{equation*}
    \conn\CC \geq \pathdom G.
  \end{equation*}
\end{observation}

We conclude this section with the definition of the induced
subcomplex. If $\CC$ is a complex on $E$ and $X\sub E$, then the
\emph{induced subcomplex} $\subc \CC X$ of $\CC$ on $X$ is the complex
on $X$ consisting of all the faces of $\CC$ contained in $X$.

%%%%%%%%%%%%%%%%%%%%%%%%%%%%%%%%%%%%%%%%%%%%%%%%%%%%%%%%%%%%%%%%%%%%%%

\section{Proof of Theorem~\ref{t:main}}
\label{sec:proof}

We now prove Theorem~\ref{t:main}. Let $k$ be a positive integer and
$G$ be a graph with $\arbf G \leq k + \eps$, where
\begin{equation*}
  \eps = \frac1{3k+2}.
\end{equation*}
Throughout this section, we write $E$ for $E(G)$, $\MM$ for the
cycle matroid of $G$ and $\CC$ for the matching complex of $G$. We
also let $\NN$ denote the matroid $(\union k \MM)^*$. Our aim is to
use Theorem~\ref{t:inters} to decompose $E$ into $k$ forests and a
matching in $G$. The following easy lemma provides the link:

\begin{lemma}\label{l:link}
  The set $E$ can be covered by $k$ forests and a matching if and only
  if there exists a base of $\NN$ which is a matching of $G$ (i.e., a
  face of $\CC$).
\end{lemma}
\begin{proof}
  We prove necessity first. Let a matching $M$ of $G$ be a base of
  $\NN$. Then $E\sm M$ is a base of $\NN^* = \union k \MM$. In
  particular, $E\sm M$ is the union of $k$ forests in $G$.

  To prove sufficiency, let $E = F\cup M'$, where $F$ is (the edge set
  of) a union of $k$ forests and $M'$ is a matching. Since $F$ is
  independent in $\union k \MM$, there is a base $B$ of $\union k \MM$
  containing $F$. Its complement $E\sm B$ is a base of $\NN$ that is
  disjoint from $F$ and hence contained in $M'$. As a subset of a
  matching, $E\sm B$ is a matching itself. This proves the lemma.
\end{proof}

We first characterize the independent sets of $\NN$. Let us say that a
set of edges $B\sub E$ is \emph{basic} if the subgraph induced by $B$
can be covered by $k$ forests and has the maximum possible number of
edges among the subgraphs of $G$ with this property. This is just
another way of saying that $B$ is a base of $\union k \MM$. From the
definition of the dual matroid, we get the following observation:

\begin{observation}\label{l:indep}
  A set $X\sub E$ is independent in $\NN$ if and only if it is
  disjoint from some basic set $B \sub E$.
\end{observation}

In Theorem~\ref{t:inters}, the restriction to sets that are
complements of flats will be crucial for us. The reason is given by
the following lemma and Lemma~\ref{l:conn} below:

\begin{lemma}\label{l:mindeg}
  If $X\sub E$ is a flat in $\NN$, then the subgraph $\eind G {E \sm
    X}$ of $G$ has minimum degree at least $k+1$.
\end{lemma}
\begin{proof}
  Suppose that $\eind G {E \sm X}$ contains a vertex $v$ of
  degree $d(v) \leq k$; let $E(v)$ denote the set of edges of $G$
  incident with $v$. Let $I$ be an inclusionwise maximal independent
  set of $\NN$ contained in $X\cap E(v)$. (Such a set exists since
  $\emptyset$ is independent.)

  Since $v$ has nonzero degree in $\eind G {E \sm X}$, we may choose
  an edge $e$ of $E(v)\sm X$. We claim that $I \cup e$ is dependent in
  $\NN$. Suppose that this is not the case. By Lemma~\ref{l:indep},
  some basic set $B\sub E$ is disjoint from $I\cup e$. Let us choose
  edge-disjoint forests $F_1,\dots,F_k$ such that $B =
  E(F_1\cup\dots\cup F_k)$.

  Since $\size{E(v)\sm X} \leq k-1$, one of the forests (say, $F_1$)
  does not contain any edge of $E(v)\sm X$. Let $F'_1$ be obtained by
  adding $e$ to $F_1$. By the maximality of $B$, $F'_1$ is not a
  forest. Thus, $F'_1$ contains a unique cycle $C$ and $e\in
  E(C)$. Let $f$ be the other edge of $C$ incident with $v$. We know
  that $f\in X$ and $f\notin I$. Since the set $I\cup f$ is disjoint
  from the set $B\cup e\sm f$ which is clearly basic, $I\cup f$ is an
  independent set of $\NN$. This contradiction with the choice of
  $I$ proves that $I\cup e$ is dependent as claimed.

  Let $I'$ be an inclusionwise minimal subset of $I$ such that $I'\cup
  e$ is dependent. Then $I'\cup e$ is a circuit of $\NN$ contained in
  $X\cup e$, contradicting Lemma~\ref{l:circuit}. This finishes the
  proof.
\end{proof}

We will now see that Lemma~\ref{l:mindeg} makes it possible to lower
bound the connectivity of the matching complexes of the subgraphs that
appear in our application of Theorem~\ref{t:inters} in terms of their
order. Without a minimum degree condition, such a bound would not be
possible, as is seen by considering the star $K_{1,n}$, whose matching
complex has $\eta = 0$ for every $n\geq 2$. The proof of the following
lemma is based on an idea due to Ond\v{r}ej Ruck\'{y} which improves
our original argument and yields a better bound.

\begin{lemma}\label{l:conn}
  Let $H$ be a graph with $\arbf H \leq k + \eps$ (where $\eps =
  1/(3k+2)$ as defined at the beginning of this section) and with
  minimum degree at least $k+1$. Let $\KK$ be the matching complex of
  $H$. Then
  \begin{equation*}
    \conn\KK \geq \eps \cdot \size{V(H)}.
  \end{equation*}
\end{lemma}
\begin{proof}
  Let $m$ and $n$ be the number of edges and vertices of $H$,
  respectively. Let $\delta$ be the minimum degree of $H$. Suppose
  that $D$ is a dominating set of 2-paths and let $H^*$ be the
  subgraph of $H$ obtained as the union of all the 2-paths in $D$. Let
  the number of vertices and edges of $H^*$ be denoted by $n^*$ and
  $m^*$, respectively. 

  Let $v$ be a vertex not contained in $H^*$. None of the (at least
  $\delta$) edges incident with $v$ is contained in $H^*$; on the other
  hand, the other endvertex of any such edge must be contained in
  $H^*$, since $D$ is dominating. This means that the number of edges
  not in $H^*$, which is $m-m^*$, can be bounded as
  \begin{equation*}
    m-m^* \geq \delta(n-n^*).
  \end{equation*}
  Together with the minimum degree assumption, this implies
  \begin{equation}
    \label{eq:both}
    (k+1)n - m \leq (k+1)n^* - m^*.
  \end{equation}
  We aim to eliminate $m$ and $n^*$ from this inequality using
  suitable bounds in terms of the other parameters.

  An easy argument (by induction on the number of paths in $D$) yields
  an upper bound on $n^*$ in terms of $m^*$, namely
  \begin{equation}
    \label{eq:star}
    n^* \leq \frac32 \cdot m^*.
  \end{equation}
  On the other hand, a bound on $m$ is implied by the assumption about
  $\arbf H$. Since
  \begin{equation*}
    k + \eps \geq \arbf H \geq \frac m {n-1} > \frac m n,
  \end{equation*}
  we obtain
  \begin{equation}
    \label{eq:nostar}
    m < n\cdot (k+\eps).
  \end{equation}

  Combining \eqref{eq:star} and \eqref{eq:nostar} with
  \eqref{eq:both}, we find
  \begin{equation*}
    m^* > n \cdot \frac{1-\eps}{\frac32(k+1)-1} = n\cdot\frac{2-2\eps}{3k+1}
    = n\cdot\frac2{3k+2} = 2\eps n.
  \end{equation*}
  Since $\size D\geq m^*/2$, the size of $D$ (an arbitrary dominating
  set of 2-paths) is at least $\eps n$. Observation~\ref{obs:pathdom}
  therefore implies that $\conn\KK \geq \eps n$ as claimed.
\end{proof}

We will need one more auxiliary result, an observation on the
definition of fractional arboricity:
\begin{lemma}\label{l:arb}
  If $G$ is a graph and $H\sub G$, then
  \begin{equation*}
    \size{E(H)} \leq \arbf G \cdot (\size{V(H)} - c(H)),
  \end{equation*}
  where $c(H)$ denotes the number of components of $H$.
\end{lemma}
\begin{proof}
  The lemma holds for a connected subgraph $H$. If $H$ has components
  $H_1,\dots,H_\ell$, then for each $i=1,\dots,\ell$, we know that
  \begin{equation*}
    \size{E(H_i)} \leq \arbf G \cdot (\size{V(H_i)}-1)
  \end{equation*}
  and the claim follows by summing these inequalities.
\end{proof}

Let us finish the proof of Theorem~\ref{t:main}. By
Lemma~\ref{l:link}, we will be done if we can find a base of $\NN$
which is a face of $\CC$ (recall that $\CC$ is the matching complex of
$G$). In view of Theorem~\ref{t:inters}, we consider an arbitrary set
$X$ which is the complement of a flat of $\NN$, and aim to verify
condition~\eqref{eq:inters}. 

Using Lemma~\ref{l:dual} and then Corollary~\ref{cor:mk} (for the
equality on the last line), the right hand side of \eqref{eq:inters}
can be rewritten as
\begin{align*}
  \rankx\NN E - \rankx\NN{E\sm X} &= \Bigl( 
  \size E + \rankx{\union k \MM}\emptyset - \rankx{\union k \MM}E
  \Bigr)\\
  &\qquad- \Bigl(
  \size {E\sm X} + \rankx{\union k \MM}X - \rankx{\union k \MM}E
  \Bigr)\\
  &= \size X - \rankx{\union k \MM}X\\
  &= \size X - \min_{T\sub X}\Bigl( \size{X\sm T} + k\cdot\rankx\MM T\Bigr).
\end{align*}
Consequently, to be able to apply Theorem~\ref{t:inters}, it suffices
to verify that
\begin{equation}\label{eq:verify}
  \conn{\subc\CC X} \geq \size T - k\cdot\rankx\MM T
\end{equation}
for every $T\sub X$.

Consider a fixed $T\sub X$. Let $n(T)$ and $n(X)$ denote the number of
vertices of $\eind G T$ and $\eind G X$, respectively, and let $c(T)$
be the number of components of $\eind G T$. By Lemma~\ref{l:rank},
$\rankx\MM T = n(T)-c(T)$.

As for the left hand side of \eqref{eq:verify}, observe that $\subc\CC
X$ is just the matching complex of $\eind G X$. Since $X$ is the
complement of a flat of $\NN$, $\eind G X$ has minimum degree at least
$k+1$ (Lemma~\ref{l:mindeg}). By Lemma~\ref{l:conn}, $\conn{\subc\CC
  X} \geq \eps n(X)$.

With the aim of establishing~\eqref{eq:verify} (for the given $X$ and
$T$) in mind, we write
\begin{align*}
  \size T &\leq (k + \eps)\cdot (n(T)-c(T))\\
  &\leq k \cdot (n(T)-c(T)) + \eps n(X)\\
  &\leq k \cdot \rankx\MM T + \conn{\subc\CC X},
\end{align*}
where we use the fractional arboricity assumption and
Lemma~\ref{l:arb} for the first inequality, the inclusion $T\sub X$
for the second one, and the above interpretation of $\rankx\MM T$
together with Lemma~\ref{l:conn} for the last one. The resulting upper
bound for $\size T$ is equivalent to \eqref{eq:verify}. Thus, we have
verified \eqref{eq:verify} for any set $X$ which is the complement of
a flat of $\NN$ and any $T\sub X$, and a final invocation of
Theorem~\ref{t:inters} completes the proof.

%%%%%%%%%%%%%%%%%%%%%%%%%%%%%%%%%%%%%%%%%%%%%%%%%%%%%%%%%%%%%%%%%%%%%%

\section*{Acknowledgment}
\label{sec:acknowledgment}

We are indebted to Ond\v{r}ej Ruck\'{y} for an improvement of the
proof of Lemma~\ref{l:conn}.

%%%%%%%%%%%%%%%%%%%%%%%%%%%%%%%%%%%%%%%%%%%%%%%%%%%%%%%%%%%%%%%%%%%%%%


\begin{thebibliography}{99}

\bibitem{AB:intersection} R. Aharoni and E. Berger, The intersection
  of a matroid and a simplicial complex, \emph{Trans. Amer. Math.
    Soc.} \textbf{358} (2006), 4895--4917.

\bibitem{ABZ:independent} 
  R. Aharoni, E. Berger and R. Ziv, Independent systems of
  representatives in weighted graphs, \emph{Combinatorica} \textbf{27}
  (2007), 253--267.

\bibitem{Bjo:topological} 
  A. Bj\"{o}rner, Topological methods, Chapter 34 in \emph{Handbook of
    Combinatorics, Vol. 2}, Elsevier, 1995, pp. 1819--1872.

\bibitem{BKL:homotopy} 
  A. Bj\"{o}rner, B. Korte and L. Lov\'{a}sz, Homotopy properties of
  greedoids, \emph{Adv. Appl. Math.} \textbf{6} (1985), 447--494.

\bibitem{BKSY:decomposing} 
  O. V. Borodin, A. V. Kostochka, N. N. Sheikh and G. Yu, Decomposing
  a planar graph with girth 9 into a forest and a matching,
  \emph{European J. Combin.} \textbf{29} (2008), 1235--1241.

\bibitem{Die:graph} 
  R. Diestel, \emph{Graph Theory}, Springer, 2005.

\bibitem{Edm:minimum} 
  J. Edmonds, Minimum partition of a matroid into independent subsets,
  \emph{J.~Res. Nat. Bur. Standards} \textbf{69} (1965) 67--72.

% \bibitem{Edm:submodular} 
%   J. Edmonds, Submodular functions, matroids and certain polyhedra,
%   in: \emph{Combinatorial structures and their applications}
%   (Proc. Calgary International Conf. 1969), pp. 69--87, Gordon and
%   Breach, New York, 1970.

\bibitem{G:covering} 
  D. Gon\c{c}alves, Covering planar graphs with forests, one having
  bounded maximum degree, \emph{J. Combin. Theory, Ser. B} \textbf{99}
  (2009), 314--322.

\bibitem{Kai:note} 
  T. Kaiser, A note on interconnecting matchings in graphs,
  \emph{Discrete Math.} \textbf{306} (2006), 2245--2250.

% \bibitem{Mes:clique} 
%   R. Meshulam, The clique complex and hypergraph matching,
%   \emph{Combinatorica} \textbf{21} (2001), 89--94.

% \bibitem{Mes:domination} 
%   R. Meshulam, Domination numbers and homology,
%   \emph{J.~Combin. Theory Ser. A} \textbf{102} (2003), 321--330.

\bibitem{MORZ:covering} 
  M. Montassier, P. Ossona de Mendez, A. Raspaud and X. Zhu,
  Decomposing a graph into forests, submitted for publication.

\bibitem{NW:decomposition}
  C. St. J. A. Nash-Williams, Decomposition of finite graphs into
  forests, \emph{J. London Math. Soc.}
  \textbf{39} (1964), 12.

\bibitem{NW:application}
  C. St. J. A. Nash-Williams, An application of matroids to graph
  theory, in: \emph{Theory of graphs --- International Symposium ---
    Th\'{e}orie de graphes --- Journ\'{e}es internationales
    d'\'{e}tude} (Rome, 1966; P. Rosenstiehl, ed.), Gordon and Breach,
  New York, and Dunod, Paris, 1967, pp. 263--265.

\bibitem{Oxl:matroid} 
  J. G. Oxley, \emph{Matroid Theory}, Oxford University Press, Oxford,
  1992.

\bibitem{P:graphes} C. Payan, Graphes \'{e}quilibr\'{e}s et
  arboricit\'{e} rationnelle, \emph{European J. Combin.} \textbf{7}
  (1986) 263--270.

\bibitem{Sch:combinatorial}
  A. Schrijver, \emph{Combinatorial Optimization}, Springer, 2003.

\end{thebibliography}
\end{document}